\documentclass[reqno, 11pt]{amsart}

\usepackage{enumitem}
\setenumerate[0]{label=(\roman*)}

\usepackage[utf8]{inputenx} 
\usepackage[T1]{fontenc}    
\usepackage[a4paper, hmargin={2.7cm,2.7cm},vmargin={3.3cm,3.3cm}]{geometry}
\usepackage{amssymb}   
\usepackage{amsthm}    
\usepackage{thmtools}  
\usepackage{mathtools} 
\usepackage{mathrsfs}  
\usepackage{commath}
\usepackage{bbm}
\usepackage[textwidth=20mm]{todonotes}

\usepackage{varioref}
\usepackage{hyperref}
\usepackage[nameinlink, capitalize, noabbrev]{cleveref}

\declaretheorem[style = plain, numberwithin = section]{theorem}

\declaretheorem[style = plain,      sibling = theorem]{lemma}
\declaretheorem[style = plain,      sibling = theorem]{proposition}
\declaretheorem[style = definition, sibling = theorem]{definition}

\declaretheorem[style = definition, sibling = theorem]{remark}

\newcommand{\N}{\mathbb{N}}

\newcommand{\C}{\mathbb{C}}

\DeclareMathOperator{\cdim}{cdim}

\newcommand{\B}{\mathcal{B}}

\DeclareMathOperator{\covol}{covol}

\def\lbdd#1{\mathrm{L}_{#1}}
\def\rbdd#1{\mathrm{R}_{#1}}
\def\vN{\mathrm{W}^*}

\newcommand{\Hi}{\mathcal{H}}

\title{Bessel duality of Gabor systems: A von Neumann algebraic perspective}

\author{Ulrik Enstad}
\address{Department of Mathematics,
University of Oslo,
Moltke Moes vei 35,
0851 Oslo.}
\email{ubenstad@math.uio.no}

\author{Franz Luef}
\address{Department of Mathematical Sciences,
Norwegian University of Sciences and Technology,
7491 Trondheim, Norway}
\email{franz.luef@ntnu.no}

\thanks{
UE was supported by The Research Council of Norway through project 314048.
}

\begin{document}

\begin{abstract}
Bessel duality of regular Gabor systems states that a Gabor system over a lattice is a Bessel sequence if and only if the corresponding Gabor system over the adjoint lattice is a Bessel sequence. We show that this fundamental result of time-frequency analysis can be deduced from a theorem in the theory of bimodules over von Neumann algebras, namely that under certain conditions, their left and right bounded vectors coincide.
\end{abstract}

\maketitle

\section{Introduction}

This paper concerns the duality theory of Gabor systems over lattices, which relates spanning properties of a Gabor system to dual spanning properties of the corresponding Gabor system over the adjoint lattice. The duality theory was established independently by Janssen \cite{Ja95}, Daubechies, H.\ J.\ Landau and Z.\ Landau \cite{DaLaLa95}, and Ron and Shen \cite{RoSh97}. The theory for general lattices in locally compact abelian groups is due to Feichtinger and Kozek \cite{FeKo98} (see also \cite{JaLe16} for duality principles for closed subgroups of the time-frequency plane, \cite{GrKo19} for an introductory account, and \cite{HaLa00,BaDuHaLaLu20,DuHaLa09} for duality principles in other contexts).

Nearly a decade earlier, Rieffel \cite{Ri88,Ri88-case} devised a method for constructing projective modules over noncommutative tori, so-called Heisenberg modules. In their most general form, these modules implement a Morita equivalence between a twisted group C*-algebra of a lattice in the time-frequency plane and a twisted group C*-algebra of its adjoint lattice. It was later discovered that Heisenberg modules provide a natural framework to interpret the duality theory of well-localized Gabor frames \cite{Lu09,Lu11,AuJaLu20}. In particular, it was shown in \cite{Lu09} (see also \cite{AuEn20}) that generators of Heisenberg modules correspond to well-localized multi-window Gabor frames, a characterization that relies fundamentally on the notion of frames in Hilbert C*-modules introduced by Frank and Larson \cite{FrLa02}. Heisenberg modules and similar module constructions have since been used to make progress on the converse of the Balian--Low theorem and its generalizations, see \cite{GrRoRo20,EnThVi25,BeEnva22,EnVi25,JaLu20}.

We shall be concerned with the most basic duality principle for Gabor systems, namely Bessel duality. Since a Gabor system with sufficiently well-localized window (e.g.\ member of the Schwartz space or Feichtinger's algebra) is automatically a Bessel sequence, Bessel duality cannot be captured by the C*-algebraic framework of Heisenberg modules. Our goal is to fill this gap and provide an analogous framework based on von Neumann algebras in which Bessel duality (and in fact the duality theory of Gabor frames with general $L^2$-windows) has a natural interpretation. Thus, we shall prove no new theorems in time-frequency analysis, but rather show how a well-known result in this field fits into a broader framework in operator algebras.

To formulate Bessel duality, fix a second-countable, locally compact group $G$ with Pontryagin dual $\widehat{G}$. Letting $g \in L^2(G)$ and $\Delta \subseteq G \times \widehat{G}$ be a lattice, we denote by $\mathcal{G}(g,\Delta)$ the corresponding Gabor system (see \Cref{sec:gabor}, \eqref{eq:gabor-system}).

\begin{theorem}[Bessel duality]\label{thm:bessel}
For a lattice $\Delta$ in $G \times \widehat{G}$ with adjoint lattice $\Delta^{\circ}$ (see \eqref{eq:adjoint-lattice}) and $g \in L^2(G)$, the following are equivalent:
\begin{enumerate}
    \item The Gabor system $\mathcal{G}(g,\Delta)$ is a Bessel sequence with bound $B$, that is,
    \[ \sum_{z \in \Delta} |\langle f, \pi(z) g \rangle |^2 \leq B \| f \|_2^2, \qquad f \in L^2(G). \]
    \item The Gabor system $\mathcal{G}(g,\Delta^{\circ})$ is a Bessel sequence with bound $\covol(\Delta)B$, that is,
    \[ \sum_{z \in \Delta^{\circ}} |\langle f, \pi(z) g \rangle |^2 \leq \covol(\Delta)B \| f \|_2^2, \qquad f \in L^2(G) . \]
\end{enumerate}
\end{theorem}

We now detail the von Neumann algebraic framework in which \Cref{thm:bessel} has a natural interpretation. This framework is provided by considering the Hilbert space $L^2(G)$ as a bimodule over twisted group von Neumann algebras of the lattices $\Delta$ and $\Delta^{\circ}$. In general, such a bimodule over von Neumann algebras equipped with traces comes with notions of right bounded vectors and left bounded vectors, see \Cref{subsec:bounded}. In our case, it turns out that $g \in L^2(G)$ is right bounded (resp.\ left bounded) if and only if the Gabor system $\mathcal{G}(g,\Delta)$ (resp.\ $\mathcal{G}(g,\Delta^{\circ})$) is a Bessel sequence (see \Cref{prop:bessel-char}). We are thus concerned with finding a suitable theorem ensuring that the left and right bounded vectors of such a bimodule coincide.

For a bifinite bimodule over factorial von Neumann algebras, it is well-known that the left and right bounded vectors coincide, see e.g.\ \cite[4.\ Proposition]{Su92}, \cite[Proposition 1.5]{Bi94}, or \cite[Proposition 5.5]{Fa09}. However, the twisted group von Neumann algebras associated with lattices in the time-frequency plane are rarely factors, so we need a more general result. Beyond the factorial case, the appropriate notion of bifiniteness is that the module has bounded center-valued von Neumann dimension (see \Cref{subsec:center-valued}) both as a left and a right module (equivalently, being finitely generated both as a left and as a right module). However, this is not enough to ensure that the left and right bounded vectors coincide, as remarked in e.g.\ \cite[Exercise 9.11]{AnPo}. One needs additionally that the centers of the two von Neumann algebras coincide, and that their traces are aligned appropriately. Under these conditions we have the following theorem. (See \Cref{subsec:bounded} for an explanation of the notation we use for bounded vectors, in particular the associated operators $\lbdd{f}$ and $\rbdd{f}$.)

\begin{theorem}\label{thm:left-right-bounded}
Let $(M,\tau)$ and $(N,\kappa)$ be tracial, separable von Neumann algebras and let $\Hi$ be an $M$-$N$-bimodule satisfying the following properties:
\begin{enumerate}
    \item $\Hi$ is finitely generated both as a left $M$-module and as a right $N$-module.
    \item $\Hi$ is faithful\footnote{See \Cref{rmk:faithful} for a version without the assumption of faithfulness.} both as a left $M$-module and as a right $N$-module, and the centers of $M$ and $N$ in $\B(\Hi)$ coincide.
\end{enumerate}
Suppose that $\tau$ and $\kappa$ are aligned in the sense of \Cref{def:alignment}. Then the left and right bounded vectors in $\Hi$ coincide. Specifically,
\begin{equation}
    \| \mathrm{R}_f \| \leq \| \cdim(_M \Hi) \cdot \cdim(\Hi_N) \| \cdot \| \mathrm{L}_f \| \label{eq:bounded-estimate}
\end{equation}
for all left bounded (equivalently right bounded) $f \in \Hi$, where $\cdim$ denotes the center-valued von Neumann dimension as defined in \Cref{subsec:center-valued}.
\end{theorem}

As we are not aware of such a statement in the literature, we spend the majority of the paper proving \Cref{thm:left-right-bounded} in \Cref{sec:bimodule}. Our strategy is inspired by Chapter 8 and Chapter 9 of the unpublished lecture notes \cite{AnPo} as well as Chapter 5 of \cite{Fa09}. We remark that a direct integral decomposition could also be used to reduce the proof to the already known factorial case, but since we have not been able to find the explicit estimate \eqref{eq:bounded-estimate} in the factorial case, we have chosen to rather prove \Cref{thm:left-right-bounded} from scratch.

In \Cref{sec:gabor} we show how \Cref{thm:bessel} follows from \Cref{thm:left-right-bounded} once basic properties of the bimodule $L^2(G)$ have been proved and its left and right bounded vectors have been characterized. We remark that to prove \Cref{thm:bessel} the full power of \Cref{thm:left-right-bounded} is not needed, as in this particular case the two von Neumann algebras $M$ and $N$ are commutants of each other (so in particular, their centers coincide, the products of their center-valued von Neumann dimensions equals $1$, and one gets the equality $\| \rbdd{f} \| = \| \lbdd{f} \|$).

\section{Bimodules over nonfactors}\label{sec:bimodule}

We will take for granted the basic terminology for von Neumann algebras, and refer to \cite{Ta02,JoSu97,AnPo}. Throughout this section $M$ and $N$ will denote separable von Neumann algebras, that is, they have separable preduals. We assume $M$ and $N$ to be equipped with normal faithful finite traces $\tau$ and $\kappa$ (just traces for short), respectively, hence they are finite von Neumann algebras. Note that we do not assume the traces to have norm equal to $1$. We denote by $L^2(M,\tau)$ the corresponding GNS Hilbert space coming from the norm $\| m \|_\tau = \tau(m^*m)^{1/2}$ on $M$, and by $\widehat{M} = \{ \widehat{m} : m \in M \}  \subseteq L^2(M,\tau)$ the image of $M$ inside $L^2(M,\tau)$ (and use similar notation for $N$). We will use the symbol $M'$ to denote the commutant of $M$ inside $\B(L^2(M))$. Then $M$ and $M'$ are related via $M' = JMJ$ where the modular conjugation operator $J \colon L^2(M,\tau) \to L^2(M,\tau)$ is given by $J(\widehat{m}) = \widehat{m^*}$, and $\tau'(a) = \tau(JaJ)$ for $a \in M'$ defines a trace on $M'$. The von Neumann algebra of $k \times k$-matrices with values in $M$ will be denoted by $\mathrm{M}_k(M)$.

\subsection{Modules}

The opposite algebra $N^{\mathrm{op}}$ consists of elements $n^{\mathrm{op}}$ for $n \in N$, but where the multiplication is defined as $n_1^{\mathrm{op}} n_2^{\mathrm{op}} = (n_2 n_1)^{\mathrm{op}}$ for $n_1, n_2 \in M$.

A Hilbert space $\Hi$ is said to be
\begin{itemize}
    \item a \emph{left $M$-module} when it comes equipped with a normal representation $\pi \colon M \to \mathcal{B}(\Hi)$.
    \item a \emph{right $N$-module} when it comes equipped with a normal representation $\rho \colon N^{\mathrm{op}} \to \B(\Hi)$;
    \item an \emph{$M$-$N$-bimodule} when it comes equipped with normal representations $\pi \colon M \to \mathcal{B}(\Hi)$ and $\rho \colon N^{\mathrm{op}} \to \mathcal{B}(\Hi)$ that commute, that is,
\[ \pi(m)\rho(n^{\mathrm{op}}) = \rho(n^{\mathrm{op}})\pi(m) \qquad \text{for all } m \in M \text{ and } n \in N . \]
\end{itemize}

Usually we suppress the normal representations and just write $mf = \pi(m)f$ and $fn = \rho(n^{\mathrm{op}})f$ depending on whether we have a left or right module. For an $M$-$N$-bimodule one then has $m(fn) = (mf)n$ for all $m \in M$, $f \in \Hi$, and $n \in N$.

When we work with a one-sided module, it will usually be a right module. All notions introduced and propositions proved for right modules will have clear analogues for left modules.

We denote by $\B(\Hi_N)$ the set of all bounded, $N$-linear maps on a right $N$-module $\Hi$, and by $\B(_M \Hi)$ the set of all bounded, $M$-linear maps on a left $M$-module $\Hi$. These are both von Neumann subalgebras of $\B(\Hi)$. For two different right $N$-modules $\Hi$ and $\mathcal{K}$, $\B(\Hi_N, \mathcal{K}_N)$ denotes the bounded, $N$-linear operators $\Hi \to \mathcal{K}$.

The Hilbert space $L^2(N,\kappa)$ is an $N$-$N$-bimodule with respect to the actions
\begin{align*}
    n_1 \widehat{n_2} = \widehat{n_1n_2}, \qquad \widehat{n_2} n_1 = \widehat{n_2n_1}, \qquad n_1,n_2 \in M .
\end{align*}
A right $N$-module $\Hi$ is called \emph{separable} if its underlying Hilbert space is, and \emph{faithful} if $fn = 0$ for all $f \in \Hi$ implies $n=0$. 

\subsection{Center-valued von Neumann dimension}\label{subsec:center-valued}

Given a von Neumann subalgebra $B$ of $N$, a normal positive linear map $E \colon N \to B$ is said to be a \emph{conditional expectation} if
\begin{enumerate}
    \item $E(b_1 n b_2) = b_1 E (n) b_2$ for all $b_1,b_2 \in B$ and $n \in N$.
    \item $E(b) = b$ for all $b \in B$.
\end{enumerate}
There exists a unique such conditional expectation with the property that $\kappa|_B \circ E = \kappa$ and we denote it by $E_B^N$. It is faithful (see \cite[Proposition 2.36]{Ta02} or \cite[Chapter 9]{AnPo}).

Denote by $Z = Z(N)$ the center of $N$. Being a commutative separable von Neumann algebra with a finite trace, there exists a finite measure space $(X,\mu)$ such that $Z \cong L^\infty(X,\mu)$. We denote by $\widehat{Z}^+$ the set of measurable extended real-valued functions on $X$, up to almost everywhere equality.

A conditional expectation $E$ from $N$ to $Z$ such that
\[ E(n^*n) = E(nn^*), \quad n \in N, \]
is called a \emph{center-valued trace}. $N$ admits a unique center-valued trace, and it is given by $E_Z^N$.

Let $\Hi$ be a separable right $N$-module. By \cite[Proposition 2.1.2]{JoSu97} there exists a projection $p \in \mathcal{B}(\ell^2(\N)) \otimes N$ such that $\Hi = p (\ell^2(\N) \otimes L^2(N,\kappa))$. One defines the \emph{center-valued von Neumann dimension} \cite{Be04} of $\Hi$ to be
\begin{equation}
    \cdim(\Hi_N) = (\mathrm{Tr} \otimes E_Z^N)(p) \label{eq:center_valued_dimension}
\end{equation}
where $E_Z^{N}$ denotes the center-valued trace of $N$ and $\mathrm{Tr}$ denotes the trace on $\B(\ell^2(\N))$ (see also the coupling operator \cite[p.\ 339 Definition 3.9]{Ta02}). The center-valued von Neumann dimension is interpreted as an (equivalence class of a) measurable nonnegative extended real-valued function on the measure space $X$ underlying $Z$. Thus, $\cdim(\Hi_N)$ is bounded if and only it is bounded almost everywhere as a function on $X$.

The center-valued von Neumann dimension competely determines an inclusion of $M$-modules in the sense that an $M$-module $\mathcal{K}$ is isomorphic to a submodule of an $M$-module $\mathcal{H}$ if and only if $\cdim_M \mathcal{K} \leq \cdim_M \Hi$, see e.g.\ \cite{Be04} or \cite[Chapter III, §4]{Di81}.

\subsection{Finitely generated modules}

A right $N$-module is called \emph{finitely generated} if there exists $g_1, \ldots, g_k \in \Hi$ such that the set $\mathrm{span}_N \{ g_1, \ldots, g_k \} = \{ \sum_{j=1}^k g_j n_j : n_j \in N \}$ is dense in $\Hi$.

\begin{proposition}\label{prop:finitely-generated-modules}
A faithful right $N$-module $\Hi$ is finitely generated if and only if $\cdim(\Hi_N)$ is bounded. If this is the case, then the following hold:
\begin{enumerate}
    \item $\tilde{N} = \B(\Hi_N)$ is a finite von Neumann algebra, and its unique center-valued trace is characterized by
    \[  \cdim(\Hi_N) \cdot E_Z^{\tilde{N}}(TT^*) = E_Z^N(T^*T), \quad T \in \B(L^2(N,\kappa)_N,\Hi_N) . \]
    \item For every trace $\kappa$ on $N$ there is a unique trace $\tau$ on $\tilde{N}$ such that
    \[ \tau( TT^*) = \kappa(T^*T), \qquad T \in \B(L^2(N,\kappa)_N, \Hi_N) . \]
    \item We have that
    \[ \cdim(\Hi_N) \cdot \cdim(_{\tilde{N}}\Hi) = 1 . \]
\end{enumerate}
\end{proposition}

\begin{proof}
Since the center-valued von Neumann dimension determines inclusion of right $N$-modules, we have that $\Hi_N$ is a submodule of $L^2(N,\kappa)^k$ if and only if $\cdim(\Hi_N) \leq \cdim(L^2(N,\kappa)^k) = k$. This shows that $\Hi_N$ is finitely generated if and only if $\cdim(\Hi_N)$ is bounded.

(i): The formula relating $E_Z^{\tilde{N}}$ and $E_Z^N$ follows from \cite[Theorem 3.8]{Ta02}, see also \cite[Section 9.3]{AnPo}.

(ii): The existence and uniqueness of $\tau$ follows from (i), see also \cite[Proposition 8.4.2]{AnPo}.

(iii): See \cite[Proposition 3.10 (i)]{Ta02}.
\end{proof}

\subsection{Restriction to a subalgebra}

This subsection deals with the case where $B$ is a von Neumann subalgebra of $N$. Given a trace $\kappa$ on $N$, the GNS space $L^2(N,\kappa)$ can be considered both a right $N$-module and a right $B$-module. We denote by $e_B \colon L^2(N,\kappa) \to L^2(N,\kappa)$ the projection onto the closed subspace $L^2(B,\kappa) \subseteq L^2(N,\kappa)$. This projection has the following relation to the conditional expectation $E_B^N$:
\begin{equation}
    e_B n e_B = E_B^N(n) e_B , \quad n \in N .
\end{equation}
The following lemma gives some essential information about the module $L^2(N,\kappa)_B$.

\begin{lemma}\label{lem:essential_lemma}
Let $(N,\tau)$ be a tracial von Neumann algebra and let $B \subseteq N$ be a von Neumann subalgebra such that $Z = Z(B) = Z(N)$. Denote by $e_B \colon L^2(N,\kappa) \to L^2(N,\kappa)$ the projection onto the closed subspace $L^2(B)$. Assume that $L^2(N,\kappa)_B$ is finitely generated. Then the following hold:
\begin{enumerate}
    \item The von Neumann algebra $\tilde{B} = \B(L^2(N,\kappa)_B)$ coincides with the von Neumann subalgebra of $\B(L^2(N,\kappa))$ generated by $N$ and $e_B$.
    \item The set
    \[ N e_B N = \{ n_1 e_B n_2 : n_1, n_2 \in N \} \]
    is ultraweakly dense in $\tilde{B}$, and the center-valued trace $E_Z^{\tilde{B}}$ on $\tilde{B}$ is determined by
    \[ \cdim(L^2(N,\kappa)_B) \cdot E_Z^{\tilde{B}}(n_1 e_B n_2) = E_Z^N (n_1 n_2), \quad n_1, n_2 \in N, \]
    where $E_Z^N$ denotes the unique center-valued tracial on $N$.
    \item For every $a \in \tilde{B}$ there exists a unique $n \in N$ such that $a e_B = n e_B$, and it is given by
\[ n = \cdim(L^2(N,\kappa)_B) \cdot E_N^{\tilde{B}}(ae_B) . \]
\end{enumerate}
\end{lemma}

\begin{proof}
For a proof of (i) see e.g. \cite[Proposition 9.4.2]{AnPo}.

For (ii), the ultraweak density of $Ne_BN$ in $\tilde{B}$ is also proved in \cite[Proposition 9.4.2 (6)]{AnPo}. Given $n_1, n_2 \in N$ we consider the bounded $B$-linear operators $T,S \colon L^2(B,\kappa) \to L^2(N,\kappa)$ given on the dense subset $\widehat{B}$ by $T(\widehat{b}) = \widehat{n_1b}$ and $S(\widehat{b}) = \widehat{n_2b}$ for $b \in B$. Then the adjoints are given by $T^*(\widehat{n}) = e_B(\widehat{n_1^*n})$ and $S^*(\widehat{n}) = e_B(\widehat{n_2^*n})$ for $n \in N$. Hence
\begin{align*}
    T^* S = E_B^N(n_1^* n_2), && ST^* = n_1 e_B n_2^* .
\end{align*}
Setting $z = \cdim(L^2(N,\kappa)_B)$, \Cref{prop:finitely-generated-modules} (i) then gives that
\begin{equation}
    z E_Z^{\tilde{B}}(n_1 e_B n_2^*) = z E_Z^{\tilde{B}}( ST^* ) = E_Z^B ( T^*S) = E_Z^B (E_B^N (n_1^* n_2)) = E_Z^N(n_1^* n_2) .
\end{equation}
Hence by normality and density the above formula completely determines $E_Z^{\tilde{B}}$.

For part (iii), note that
\[ b E_N^{\tilde{B}}(e_B) = E_N^{\tilde{B}}(be_B) = E_N^{\tilde{B}}(e_B b) = E_N^{\tilde{B}}(e_B) b, \qquad b \in B. \]
Hence $E_N^{\tilde{B}}(e_B) \in Z(B) = Z(N)$, so using part (ii) we get that
\[ z E_N^{\tilde{B}}(e_B) = z E_Z^{N}(E_N^{\tilde{B}}(e_B)) = z E_Z^{\tilde{B}}(e_B) = E_Z^N(1) = 1. \]
Now let $a \in \tilde{B}$ suppose that $ae_B = ne_B$ for some $n \in N$. Then
\[ zE_N^{\tilde{B}}(ae_B) = zE_N^{\tilde{B}}(ne_B) = n z E_N^{\tilde{B}}(e_B) = n, \]
which shows the required uniqueness. For existence, suppose first that $a = n_1 e_B n_2$ for some $n_1, n_2 \in N$. Then
\[ ae_B = n_1 e_B n_2 e_B = n_1 E_B^N(n_2) e_B. \]
Since $n_1 E_B^N(n_2) \in N$ we conclude by the established uniqueness that $n_1 E_B^N(n_2) = z E_N^{\tilde{B}}(ae_B)$. More generally we can, using part (i), write $a \in \tilde{B}$ as an ultraweak limit of a net $a_i$ of elements from $Ne_BN$. As we have already shown $a_i e_B = z E_N^{\tilde{B}}(a_i e_B)e_B$. Since multiplication is separately ultraweakly continuous and $E_N^{\tilde{B}}$ is normal, it follows that
\[ a e_B = \lim_i (a_i e_B) = \lim_i (z E_N^{\tilde{B}}(a_i e_B)e_B) = z E_N^{\tilde{B}}(\lim_i a_i e_B)e_B = z E_N^{\tilde{B}}(a e_B)e_B . \]
This finishes the proof.
\end{proof}

If $\Hi$ is a right $N$-module, it may also be considered a right $B$-module, and the following proposition relates the center-valued von Neumann dimensions of these two modules given that the centers of $N$ and $B$ coincide.

\begin{proposition}\label{prop:coefficient-change}
If $\Hi$ is a faithful, finitely generated right $N$-module and $B \subseteq N$ is a von Neumann subalgebra such that $Z(B) = Z(N)$, then $\Hi_B$ is finitely generated if and only if $L^2(N,\kappa)_B$ is finitely generated, and in that case
    \[ \cdim({\Hi_B}) = \cdim({L^2(N,\kappa)}_B) \cdot \cdim({\Hi_N}) . \]
\end{proposition}

\begin{proof}
If $\mathcal{K}$ and $\mathcal{K}'$ are faithful finitely generated $N$-modules, then by \cite[Proposition 2.1.2]{JoSu97} $\mathcal{K}$ is isomorphic to a submodule of $\mathcal{K}' \otimes \C^k$ for some $k \in \N$. But then $\cdim(\mathcal{K}_B) \leq \cdim((\mathcal{K}' \otimes \C^k)_B) = k \cdim(\mathcal{K}'_B)$, so reversing the roles of $\mathcal{K}$ and $\mathcal{K}
'$, this shows that $\cdim(\mathcal{K}_B)$ is bounded if and only if $\cdim(\mathcal{K}'_N)$ is bounded. We may apply the above argument to $\mathcal{K} = \Hi$, $\mathcal{K}' = L^2(N,\kappa)$, and use \Cref{prop:finitely-generated-modules} to deduce that $\Hi_B$ is finitely generated if and only if $L^2(N,\kappa)_B$ is.

Denote by $Z$ the common center of $N$ and $B$. By \cite[Proposition 2.1.2]{JoSu97} we may assume that $\Hi_N =  p L^2(N)^k$ for some projection $p \in \mathrm{M}_k(N)$ which has central support equal to $1$ since $\Hi_N$ is faithful. Letting $\tilde{B}$ denote the commutant of $B$ in $\B(\Hi)$, \cite[Proposition 3.10 (ii)]{Ta02} gives
\begin{align*}
    \cdim(\Hi_B) = E_Z^{\tilde{B}}(p) \cdim(L^2(N,\kappa)^k_B) = k E_Z^{\tilde{B}}(p) \cdim(L^2(N,\kappa)_B).
\end{align*}
Now $\cdim(\Hi) = (\mathrm{Tr} \otimes E_Z^N)(p) = k E_Z^{p\mathrm{M}_k(N)p}(p)$. Since $p \mathrm{M}_k(N)p \subseteq \tilde{B}$ and $Z(p \mathrm{M}_k(N)p) = Z = Z(\tilde{B})$, we deduce that $E_Z^{p\mathrm{M}_k(N)p}(p) = E_Z^{\tilde{B}}(p)$ which this finishes the proof.
\end{proof}

\subsection{Bounded vectors}\label{subsec:bounded}

A vector $f$ in a left $M$-bimodule $\Hi$ is \emph{right $\tau$-bounded} if there exists $C > 0$ such that
\[ \| m f \| \leq C \tau(mm^*)^{1/2} , \qquad m \in M . \]
If $\Hi$ if a right $N$-module, a vector $f \in \Hi$ is called \emph{left $\kappa$-bounded} if there exists $C > 0$ such that
\[ \| f n \| \leq C \kappa(nn^*)^{1/2}, \qquad n \in N. \]
In an $M$-$N$-bimodule one has notions of both right $\tau$-bounded vectors and left $\kappa$-bounded vectors. If $f \in \Hi$ is right $\tau$-bounded, then the assignment $\widehat{M} \to \Hi$, $\widehat{m} \mapsto mf$ extends to a bounded linear operator $\rbdd{f} \colon L^2(M,\tau) \to \Hi$. Similarly, if $f$ is left $\kappa$-bounded, the assignment $\widehat{N} \to \Hi$, $\widehat{n} \mapsto fn$ extends to a bounded linear operator $\lbdd{f} \colon L^2(N,\kappa) \to \Hi$.

Note that the property of being left or right bounded depends on the chosen traces on $M$ and $N$, so to meaningfully have a relation between the two notions, we must impose some compatibility between $\tau$ and $\kappa$. We thus make the following definition:

\begin{definition}\label{def:alignment}
We say that $\tau$ and $\kappa$ are \emph{aligned on $\Hi$} if whenever $T \colon L^2(N, \kappa) \to \Hi$ is bounded and $N$-linear and $TT^* \in M$, then
\[ \tau(TT^*) = \kappa(T^*T) . \]
Equivalently, whenever $S \colon L^2(M,\tau) \to \Hi$ is bounded and $M$-linear and $SS^* \in N^{\text{op}}$, then
\[ \tau'(S^*S) = \kappa(SS^*). \]
\end{definition}

Note that according to this definition, the traces $\tau$ on $\tilde{N}$ and $\kappa$ on $N$ in \Cref{prop:finitely-generated-modules} (ii) are aligned on $\Hi$. One may then obtain a trace on $M$ aligned with $\kappa$ on $N$ by restricting $\tau$ to $M$.

\begin{lemma}\label{lem:bounded-commutant-module}
Let $\Hi$ be a finitely generated right $N$-module and let $\tilde{N} = \B(\Hi_N)$. For a trace $\kappa$ on $N$ and the corresponding trace $\tau$ on $\tilde{N}$ from \Cref{prop:finitely-generated-modules} (ii), the left $\kappa$-bounded vectors and the right $\tau$-bounded vectors coincide. Specifically,
\[ \| \mathrm{L}_f \| = \| \mathrm{R}_f \| \]
for all left bounded (equivalently right bounded) $f \in \Hi$.
\end{lemma}

\begin{proof}
By \cite[Proposition 2.1.2]{JoSu97} we may assume that $\Hi_N = pL^2(N,\kappa)^k$ for $p$ a projection in $\mathrm{M}_k(N)$. Then $\tilde{N} = p \mathrm{M}_k(N) p \subseteq \B(pL^2(N,\kappa)^k)$.

Let $f = (f_1, \ldots, f_k) \in \Hi$. If $f$ is left $\kappa$-bounded, then in particular each $f_i$ is a left $\kappa$-bounded element of $L^2(N,\kappa)$, hence in $\widehat{N}$ (see \cite[Theorem 1.2.4]{JoSu97}). On the other hand, suppose that $f$ is right $\tau$-bounded. Using a Pimsner--Popa basis, see \cite{PiPo86} or \cite[Lemma 8.4.8]{AnPo}, we may pick a partial isometry $v \in M_k(N)$ such that $v^*v = p$ and $vv^*=q$ is a diagonal matrix with projections $p_1, \ldots, p_r \in N$ on the diagonal. Then $p M_k(N) p \cong q M_k(N) q$ and $pL^2(N)^k \cong qL^2(N)^k = p_1 L^2(N) \oplus \cdots \oplus p_k L^2(N)$, so it suffices to prove the boundedness in this case. Let $n \in p_iNp_i$ and take $A \in qM_k(N)q$ to be $A_{i,i} = n$ and $A_{i,j} = 0$ otherwise. Then $A f = (0, \ldots, n f_i, \ldots, 0)$, so left $\tau$-boundedness of $f$ gives
\[ \| n f_i \|^2 \leq \tau(A^*A) = \kappa(n^*n), \quad n \in p_i N p_i .  \]
This means that $f_i \in p_i L^2(N)$ is right bounded for the left action of $p_i N p_i$, hence we must have $f_i = \widehat{n_i} \in p_i \widehat{N}$. Thus $f \in p \widehat{N}^k$. This shows that the right and left bounded vectors agree.

Next we characterize $\tau$. Let $f = (\widehat{n}_1, \ldots, \widehat{n}_k) \in p \widehat{N}^k$ and set $T = \lbdd{f}$. Then $TT^*$ is the matrix $( n_i n_j^* )_{i,j=1}^k \in p \mathrm{M}_k(N) p$, and
\[ \tau(TT^*) = \kappa(T^*T) = \sum_{j=1}^k \kappa(n_j^* n_j) . \]
Writing a general element in $p \mathrm{M}_k(N) p$ as a sum of elements of the form $TS^*$ where $T$ and $S$ are of the form above and using polarization, this shows that $\tau$ is given by $\kappa \otimes \mathrm{Tr}$.

We now compute for $f = ( \hat{n}_1, \ldots, \hat{n}_k) \in \Hi$ that
\begin{align*}
    \| f n \|^2 &= \sum_{j=1}^k \kappa( (n_k n) (n_k n)^* ) = \kappa \big( n n^* \sum_{j=1}^k n_j^* n_j \big) = \kappa \big( \big( \sum_{j=1}^k n_j^* n_j \big)^{1/2} n n^*) .
\end{align*}
Hence $\| \lbdd{f} \|$ is the operator norm of the map $L^2(N,\kappa) \to L^2(N,\kappa)$ given by left multiplication with $\big( \sum_{j=1}^k n_j^* n_j \big)^{1/2}$, so
\[ \| \lbdd{f} \| = \Big\| \sum_{j=1}^k n_j^* n_j \Big\|^{1/2} . \]
On the other hand, the adjoint of $\rbdd{f}$ is given by $\rbdd{f}^*(g) = A = ( \widehat{n_i' n_j^*} )_{i,j=1}^k$ for $g = (\hat{n}_1', \ldots, \hat{n}_k') \in p \widehat{N}^k$. Thus
\begin{align*}
    \| \rbdd{f}(g) \|^2 &= \tau \big(AA^*) \\
    &= \sum_{i=1}^k \kappa \Big( \sum_{r=1}^k n_i' n_r^* n_r n_i'^* \Big) \\
    &= \kappa\Big( \big( \sum_{i=1}^k n_i'^* n_i \big) \big( \sum_{r=1}^k n_r^* n_r \big) \Big) .
\end{align*}
Then similarly we conclude that $\| \rbdd{f} \| = \| \sum_{r=1}^k n_r^* n_r \| = \| \lbdd{f} \|$.
\end{proof}

We now consider the case where $B$ is a von Neumann subalgebra of $N$ and $\Hi = L^2(N,\kappa)_N$. For a left $\kappa$-bounded vector $f \in \Hi_N$ we write $\lbdd{f}^N \colon L^2(N,\kappa) \to \Hi$ for the corresponding map, and if $f$ is left $\kappa|_B$-bounded in $\Hi_B$, we write $\lbdd{f}^B \colon L^2(B,\kappa) \to \Hi$.

\begin{lemma}\label{lem:bounded_vectors_subalgebra}
Let $B \subseteq N$ be a von Neumann subalgebra such that $Z(B) = Z(N)$ and such that $L^2(N,\kappa)_B$ is finitely generated. Then the left $\kappa$-bounded vectors of $L^2(N,\kappa)_N$ and the left $\kappa|_B$-bounded vectors of $L^2(N,\kappa)_B$ coincide. Specifically
\[ \| \lbdd{f}^N \| \leq \| \cdim(L^2(N,\kappa)_B) \| \cdot \| \lbdd{f}^B \| , \qquad f \in \Hi \; \text{left $\kappa|_B$-bounded}. \]
\end{lemma}

\begin{proof}
Let $f \in L^2(N,\kappa)_B$ be left $\kappa|_B$-bounded and denote by $\lbdd{f}^B \colon L^2(B,\kappa) \to L^2(N,\kappa)$ the corresponding bounded operator. Then $\lbdd{f}^B e_B \in \tilde{B}$, so by \Cref{lem:essential_lemma} (iii) we have that $\lbdd{f}^B e_B = (\lbdd{f}^B e_B)e_B = ne_B$ where $n = \cdim(L^2(N,\kappa)_B) \cdot E_N^{\tilde{B}}( \lbdd{f}^B e_B)$. It follows that
\[ f = \lbdd{f}^B(\widehat{1}) = \lbdd{f}^B(e_B(\widehat{1})) = n \widehat{1} = \widehat{n}. \]
This shows that $f \in \widehat{N}$, i.e.,\ $f$ is left $\kappa$-bounded, and
\[ \| \lbdd{f}^N \| = \| n \| \leq \| \cdim(L^2(N,\kappa)_B) \| \cdot \| E_N^{\tilde{B}}( \lbdd{f}^B e_B) \| \leq \| \cdim(L^2(N,\kappa)_B) \| \| \lbdd{f}^B \| . \]
\end{proof}

We now prove the second main result of the introduction.

\begin{proof}[Proof of \Cref{thm:left-right-bounded}]
Since $\Hi$ is faithful both as a left $M$-module and as a right $N$-module, we may identify $M$ and $N^{\text{op}}$ with their images inside $\B(\Hi)$. We have that $M \subseteq \tilde{N} \coloneqq \B(\Hi_N)$ and $Z(M) = Z(N^{\text{op}}) = Z(\tilde{N})$. Since $\Hi_N$ is finitely generated, it follows that $_{\tilde{N}} \Hi$ is finitely generated. Since also $_M \Hi$ is finitely generated, we can apply \Cref{prop:coefficient-change} to obtain
\[ \cdim(_M \Hi) = \cdim(_M L^2(\tilde{N})) \cdot \cdim(_{\tilde{N}} \Hi) . \]
On the other hand \Cref{prop:finitely-generated-modules} (iii) gives that
\[  \cdim(_{\tilde{N}} \Hi) \cdot \cdim( \Hi_N) = 1. \]
Combining the two, we obtain
\[  \cdim(_M \Hi) \cdot \cdim(\Hi_N) = \cdim(_M L^2(\tilde{N})) . \]
In particular $\cdim(_M L^2(\tilde{N}))$ is bounded, hence $ _M L^2(\tilde{N})$ is finitely generated by \Cref{prop:finitely-generated-modules}. Equipping $\tilde{N}$ with the trace $\tilde{\tau}$ aligned with $\kappa$ from \Cref{lem:bounded-commutant-module}, the version of \Cref{lem:bounded_vectors_subalgebra} for left modules implies that the right $\kappa$-bounded vectors of $\Hi$ coincide with the right $\tilde{\tau}$-bounded vectors of $\Hi$, and that the corresponding operator norms of $f \in \Hi$ are related via
\begin{align*}
     \| \rbdd{f}^{\tilde{N}} \| &\leq \| \cdim(_M {L^2(\tilde{N})}) \| \cdot \| \rbdd{f}^{M} \| \\
     &= \| \cdim(_M \Hi) \cdot \cdim(\Hi_N) \| \cdot \| \rbdd{f}^M \| .
\end{align*}
Since $\tilde{\tau}|_{M} = \tau$ by assumption, \Cref{lem:bounded-commutant-module} implies that $\| \rbdd{f}^{\tilde{N}} \| = \| \lbdd{f}^{N} \|$, so this finishes the proof.
\end{proof}

\begin{remark}
If $M$ and $N$ are factors, then the assumption that their centers coincide in $\B(\Hi)$ holds automatically. They also have unique traces up to scaling. Thus, no assumption on the traces is needed for the mere statement that the left and bounded vectors coincide in this case, although to obtain the inequality \eqref{eq:bounded-estimate}, an assumption of alignment of the traces is needed.
\end{remark}

\begin{remark}\label{rmk:faithful}
We assume in \Cref{thm:left-right-bounded} that $\Hi$ is faithful both as a left $M$-module and as a right $N$-module. Without this assumption, one can do the following: Let $\pi \colon M \to \B(\Hi)$ and $\rho \colon N^{\mathrm{op}} \to \B(\Hi)$ denote the normal homomorphisms implementing the module actions. Consider the von Neumann algebras $M_0 = \pi(M)$ and $N_0 = \rho(N^{\mathrm{op}})^{\mathrm{op}}$. Then $\Hi$ is an $M_0$-$N_0$-bimodule which is faithful both as a left and right module, and $_M \Hi$ (resp.\ $\Hi_N$)  is finitely generated if and only if $_{M_0} \Hi$ (resp.\ $\Hi_{N_0}$) is finitely generated. If one assumes $Z(M_0) = Z(N_0)$ and aligned traces on $M_0$ and $N_0$, then \Cref{thm:left-right-bounded} gives that
\[ \| \rbdd{f} \| \leq \| \cdim(_{M_0} \Hi) \cdot \cdim(\Hi_{N_0}) \| \cdot \| \lbdd{f} \| . \]
\end{remark}

\section{Gabor systems}\label{sec:gabor}

Let $G$ denote a second-countable, locally compact abelian group with Pontryagin dual $\widehat{G}$. Fix a lattice (i.e.,\ discrete, compact subgroup) $\Delta$ in $G \times \widehat{G}$. The following defines a 2-cocycle $c$ on $G\times \widehat{G}$:
\[ c((x,\omega),(x',\omega')) = \overline{\omega'(x)}, \qquad (x,\omega),(x',\omega') \in G \times \widehat{G} . \]

We consider the twisted group von Neumann algebra $M = \vN(G,c)$. This is the von Neumann algebra generated by the $c$-projective left regular representation $\lambda_c \colon \Delta \to \mathcal{U}(\ell^2(\Delta))$ of $\Delta$ on $\ell^2(\Delta)$ given by
\[ \lambda_c(z) \delta_{z'} = c(z,z') \delta_{zz'}, \qquad z,z' \in \Delta. \]
We refer to \cite{Pa08} for a survey on 2-cocycles on groups and associated operator algebras.

The \emph{adjoint lattice} of $\Delta$ is the following subset of $G \times \widehat{G}$, which is itself a lattice in $G \times \widehat{G}$:
\begin{equation}
    \Delta^{\circ} = \{ z \in \Delta : \pi(z)\pi(w) = \pi(w)\pi(z) \text{ for all } w \in \Delta \} . \label{eq:adjoint-lattice}
\end{equation}
We consider also the von Neumann algebra $N = \vN(\Delta^{\circ}, c^{\mathrm{op}})$ where $c^{\mathrm{op}}(z,w) = c(w,z)$ for $z,w \in G \times \widehat{G}$.

Given $(x,\omega) \in G \times \widehat{G}$, define the unitary operator $\pi(x,\omega)$ on $L^2(G)$ via
\[ \pi(x,\omega)f(t) = \omega(t) f(x^{-1}t), \qquad f \in L^2(G), t \in G. \]

The following is proved in \cite[Theorem 2.16]{Ri88}, see also \cite[p.\ 336]{Be04}.

\begin{proposition}\label{prop:extension_gabor}
For any lattice $\Delta$ in $G \times \widehat{G}$ the map $\C(\Delta,c) \to \B(L^2(G))$ given by $\delta_z \mapsto \pi(z)$ for $z \in \Delta$ extends to a faithful normal representation $\pi_\Delta \colon \vN(\Delta,c) \to \B(L^2(G))$. Moreover, the commutant of $\pi_\Delta(\Delta)$ in $\B(L^2(G))$ is exactly the von Neumann algebra generated by $\pi_{\Delta^\circ}(\Delta^\circ)$.
\end{proposition}

Applying \Cref{prop:extension_gabor} to the adjoint lattice we obtain a faithful normal representation $\vN(\Delta^{\circ}, c) \to \B(\Hi)$. Formally this is equivalent to having a faithful normal representation $\vN(\Delta^\circ, c)^{\mathrm{op}} \to \B(\overline{\Hi})$ where $\overline{\Hi}$ denotes the complex conjugate space. Identifying this with $L^2(G)$ using pointwise complex conjugation and identifying $\vN(\Delta^\circ, c)^{\mathrm{op}}$ with $\vN(\Delta^\circ, c^{\mathrm{op}})$, we can view this representation as a right action on $L^2(G)$ given by
\[ f \cdot b = \overline{  \Big( \sum_{z \in \Delta^\circ} b(z) \pi(z) \Big)^* f } = \sum_{z \in \Delta^{\circ}} b(z) \overline{ \pi(z)^* f } . \]
By \Cref{prop:extension_gabor} these two actions turn $L^2(G)$ into an $M$-$N$-bimodule. For the computation of the center-valued von Neumann dimension of this module, see \cite[Theorem 3.4]{Ri88} or \cite{Be04}.

\begin{proposition}\label{prop:cdim-gabor}
We have that
\[ \cdim_M \Hi = \covol(\Delta) I . \]
\end{proposition}

We fix the tracial state $\tau$ on $M$ given by $\tau(\delta_z) = \delta_{z,0}$ for $z \in \Delta$. The unique trace $\kappa$ on $N$ that is aligned with $\tau$ is given by $\kappa ( \delta_z ) = \covol(\Delta) \delta_{z,0}$ for $z \in \Delta^{\circ}$, see \cite[p.\ 278]{Ri88}.

Given a lattice $\Delta \subseteq G \times \widehat{G}$ and $g \in L^2(G)$, the corresponding Gabor system is denoted by
\begin{equation}
    \mathcal{G}(g,\Delta) = ( \pi(z) g )_{z \in \Delta} . \label{eq:gabor-system}
\end{equation}
Recall that $\mathcal{G}(g,\Delta)$ is said to be a \emph{Bessel sequence} with bound $B$ if
\[ \sum_{z \in \Delta} | \langle f, \pi(z) g \rangle |^2 \leq B \| f \|_2^2, \quad f \in L^2(G). \]

\begin{proposition}\label{prop:bessel-char}
The following hold for $g \in L^2(G)$:
\begin{enumerate}
    \item $g$ is right $\tau$-bounded with $\| \rbdd{f} \|^2\leq B$ if and only if $\mathcal{G}(g,\Delta)$ is a Bessel sequence with bound $B$.
    \item $g$ is left $\kappa$-bounded with $\| \lbdd{g} \|^2 \leq B'$ if and only if $\mathcal{G}(g,\Delta^{\circ})$ is a Bessel sequence with bound $\covol(\Delta) B'$.
\end{enumerate}
\end{proposition}

\begin{proof}
Note that $L^2(M,\tau) \cong \ell^2(\Delta)$ with the counting measure, so $g$ is right $\tau$-bounded with $\| \rbdd{g} \|^2 \leq B$ if and only if
\[ \Big\| \sum_{z \in \Delta} a(z) \pi(z) g \Big\|^2 \leq B \| a \|_2^2 , \quad a \in \ell^2(\Delta) . \]
This means exactly that the Gabor system $\mathcal{G}(g,\Delta)$ has upper Riesz bound $B$, equivalently is a Bessel sequence with bound $B$ by \cite[Theorem 3.2.3]{Ch03}. Moreover, $L^2(N,\kappa) \cong \ell^2(\Delta^{\circ})$ but due to the normalization of $\kappa$ this $\ell^2$-space is equipped with the counting measure weighted by $\covol(\Delta)$. Hence $g$ is left $\kappa$-bounded with $\| \lbdd{g} \|^2 \leq B'$ if and only if
\[ \Big \| \sum_{z \in \Delta^{\circ}} b(z) \pi(z) g \Big\|^2 \leq \covol(\Delta)B' \| b \|_2^2 , \quad b \in \ell^2(\Delta^{\circ}) , \]
i.e.,\ precisely when $\mathcal{G}(g,\Delta^{\circ})$ is a Bessel sequence with bound $\covol(\Delta)B'$.
\end{proof}

\begin{remark}
Bessel duality holds more generally for Gabor systems along closed, cocompact subgroups of the time-frequency plane, cf.\ \cite[Theorem 6.4]{JaLe16}. In this case, with the appropriate normalization of measures on $\Delta$ and $\Delta^{\circ}$, the Bessel bounds of the two Gabor systems are identical. Note however that we use the counting measure on both in \Cref{prop:bessel-char}, which is why the constant $\covol(\Delta)$ appears in the Bessel bound of the Gabor system over the adjoint lattice.
\end{remark}

We may now derive the Bessel duality from \Cref{thm:left-right-bounded}.

\begin{proof}[Proof of \Cref{thm:bessel}]
By \Cref{prop:extension_gabor} the $M$-$N$-bimodule $L^2(G)$ is faithful both as a left module and as a right module. By \Cref{prop:cdim-gabor} it is finitely generated as a left module. By \Cref{prop:finitely-generated-modules} (iii) it is also finitely generated as a right module, and $\cdim(_M \Hi) \cdot \cdim(\Hi_N) = 1$. Using \Cref{thm:left-right-bounded}, we have that the left $\kappa$-bounded vectors and the right $\tau$-bounded vectors of this module coincide, with $\| \rbdd{f} \| = \| \lbdd{f} \|$ for all bounded $f \in L^2(G)$. Using the characterization provided by \Cref{prop:bessel-char}, this finishes the proof.
\end{proof}

\bibliographystyle{abbrv}
\bibliography{main}

\end{document}